\pdfoutput=1
\RequirePackage{ifpdf}
\ifpdf 
\documentclass[pdftex]{sigma}
\else
\documentclass{sigma}
\fi

\numberwithin{equation}{section}

\newtheorem{Theorem}{Theorem}[section]
\newtheorem{Corollary}[Theorem]{Corollary}
\newtheorem{Lemma}[Theorem]{Lemma}

{\theoremstyle{definition}

 }

\begin{document}
\allowdisplaybreaks

\renewcommand{\PaperNumber}{026}

\FirstPageHeading

\ShortArticleName{Cohomology of $\mathfrak{sl}_3$ and $\mathfrak{gl}_3$ with Coefficients in Simple Modules}

\ArticleName{Cohomology of $\boldsymbol{\mathfrak{sl}_3}$ and $\boldsymbol{\mathfrak{gl}_3}$ with Coefficients in Simple\\ Modules and Weyl Modules in Positive Characteristics}

\Author{Sherali Sh.~IBRAEV}

\AuthorNameForHeading{Sh.Sh.~Ibraev}

\Address{Korkyt Ata Kyzylorda University, Aiteke bie St.,~29A, 120014, Kzylorda, Kazakhstan}
\Email{\href{mailto:ibrayevsheraly@gmail.com}{ibrayevsheraly@gmail.com}}

\ArticleDates{Received August 12, 2021, in final form March 26, 2022; Published online March 30, 2022}

\Abstract{We calculate the cohomology of $\mathfrak{sl}_3(k)$ over an algebraically closed field $k$ of characteristic $p>3$ with coefficients in simple modules and Weyl modules. We also give descriptions of the corresponding cohomology of $\mathfrak{gl}_3(k)$.}


\Classification{17B20; 17B45; 20G05}

\section{Introduction}
There are many remarkable results in the cohomology theory of modular Lie algebras. Any simple module over the restricted Lie algebra with nontrivial cohomology is restricted, see \cite[Theorem~2]{D84}. Modules over a Lie algebra with nonzero cohomology are called {\it peculiar}. For any finite-dimensional Lie algebra, the number of non-isomorphic peculiar indecomposable modules is finite, see \cite[Theorem~1]{D85}.

Now, let $\mathfrak{g}=\mathfrak{sl}_3(k)$. The cohomology with coefficients in simple $\mathfrak{g}$~-modules is completely described only in small characteristics $p=2,3$, see \cite[Theorem~1]{IT21}, \cite[Theorem~1]{IIE21}. In~the case where $p>3$, the cohomology of simple $\mathfrak{g}$-mo\-dules are known in the following cases: for~$H^1(\mathfrak{g},M)$, see \cite[p.~301]{Jan91}, for $H^2(\mathfrak{g},M)$, see \cite[Theorem~1.1]{DI02}. In~small degrees, the cohomology has the following interpretations: $H^1({\mathfrak{g}},{\mathfrak{g}})$ is identified with the space of outer differentiation and $H^2({\mathfrak{g}},{\mathfrak{g}})$ is identified with the space of local deformations of the Lie algebra~$\mathfrak{g}$. It is known that these spaces are trivial, see \cite[p.~124]{Perm05}, \cite[p.~125]{Che05}, \cite[Lemma~2.2.1b]{BGLL}, \cite[p.~32]{BGL}. The cohomology of ${\mathfrak{g}}$ with coefficients in the trivial module is also known, see \cite[p.~42]{BGL}. In~other cases, the interpretation of the cohomology with coefficients in simple modules remains open. In~this paper, we give a complete description of the cohomology of $\mathfrak{sl}_3(k)$ with coefficients in the simple modules for $p>3$.

\subsection{Notation}
Let $\mathfrak{g}=\mathfrak{sl}_3(k)$ over an
algebraically closed field $k$ of characteristic
$p>3$ and $M$ a simple $\mathfrak{g}$-mo\-dule. Let $L(r,s)$ denote a simple $\mathfrak{g}$-module with the highest weight
$r\omega_1+s\omega_2$, where $\omega_1$, $\omega_2$ are fundamental weights.

Let $G={\rm SL}_3(k)$; it is an algebraic group, its Lie algebra is $\mathfrak{sl}_3(k)$. We will consider cohomology of $\mathfrak{sl}_3(k)$ as $G$-modules. Let $V$ be a $G$-module and $M$ be a simple $G$-module. We define a~{\it composition coefficient} $[V:M]$ for $M$ from the formula
\begin{gather*}
\operatorname{ch}(V)=\sum_{M \ \text{is simple}}[V:M]\operatorname{ch}(M),
\end{gather*}
where $\operatorname{ch}(V)$ is the formal character of the $G$-module $V$.
If $[V:M]\neq 0$, then we say that $M$ is a {\it composition factor of $V$.}

For a vector space $L$ over $k$, we denote by $L^{(1)}$ the vector space over $k$ that coincides with $L$ as an additive group and with the scalar multiplication given by
\begin{gather*}
a\cdot v=\sqrt[p]{a}v\qquad\text{for all}\quad a\in k, \quad v\in L,
\end{gather*}
where the left hand side is the new multiplication and the right hand side the old one. If $L$ is a~$G$-module, then $L^{(1)}$ is also a
$G$-module using the given action of any $g\in G$ on the additive group $L^{(1)}=L$. The new $G$-module $L^{(1)}$ is called the {\it Frobenius twist} of $L$. We define {\it higher Frobenius twists} inductively:
$L^{(d+1)}=\big(L^{(d)}\big)^{(1)}$. To each weight $\mu$ of the space $L$ there corresponds the weight $p^d\mu$ of the space $L^{(d)}$.

A weight $r\omega_1+s\omega_2$ is {\it restricted} if $0\leq r, s\leq p-1$. The composition factors of $H^n(\mathfrak{g},M)$ are Frobenius twists of some simple $G$-modules with restricted highest weights.

A $G$-module $L$ is {\it rational} if the corresponding representation is a homomorphism from $G$ to~${\rm GL}(L)$. Suppose $V$ is the Frobenius twist of some rational $G$-module. Then, there is a unique $d>0$ and rational $G$-module $L$ such that $L^{(d)}=V$. Denote this module by $V^{(-d)}$.

We will usually denote $H^n(\mathfrak{g},k)$ by $H^n(\mathfrak{g})$, and use the following short notation:
\begin{gather*}
mV:=V\oplus\cdots\oplus V\qquad (m\text{ summands}),
\end{gather*}
where $V$ is a $G$-module.

\subsection{Main result}

In this paper, $k$ is always algebraically closed field $k$ of characteristic
$p>3$.
\begin{Theorem}\label{th1} Let $\mathfrak{g}=\mathfrak{sl}_3(k)$ and $M$ be a simple $\mathfrak{g}$-module. Then, the following isomorphisms of $G$-modules hold:

\begin{enumerate}\itemsep=0pt
\item[$(a)$] $H^n(\mathfrak{g})\cong k$ for $n=0,3,5,8;$

\item[$(b)$] $H^n(\mathfrak{g},L(p-2,1))\cong
\begin{cases}
L(1,0)^{(1)}&\text{if}\ n=1,7,
\\
2L(1,0)^{(1)}&\text{if}\ n=4;\end{cases}$

\item[$(c)$] $H^n(\mathfrak{g},L(1,p-2))\cong
\begin{cases}
L(0,1)^{(1)}&\text{if}\ n=1,7,
\\
2L(0,1)^{(1)}&\text{if}\ n=4;
\end{cases}$

\item[$(d)$] $H^n(\mathfrak{g},L(p-3,0))\cong L(1,0)^{(1)}$ for $n=2,3,5,6;$

\item[$(e)$] $H^n(\mathfrak{g},L(0,p-3))\cong L(0,1)^{(1)}$ for $n=2,3,5,6;$

\item[$(f)$] $H^n(\mathfrak{g},L(p-2,p-2))\cong
\begin{cases}
k&\text{if}\ n=1,7,
\\
L(1,1)^{(1)}&\text{if}\ n=3,5,
\\
2L(1,1)^{(1)}\oplus 2k&\text{if}\ n=4.
\end{cases}$
\end{enumerate}
Otherwise, $H^n(\mathfrak{g},M)=0$.
\end{Theorem}

This theorem completes the description of the cohomology of $\mathfrak{sl}_3(k)$ with coefficients in simple modules over an algebraically closed fields of positive characteristics.

\subsection{Some applications of the main result}

Using Theorem~\ref{th1}, one can easily describe the cohomology of $\mathfrak{gl}_3(k)$ with coefficients in simple modules. Let $M$ be an $\mathfrak{sl}_3(k)$-module.
Since $\mathfrak{gl}_3(k)\cong \mathfrak{sl}_3(k)\oplus I$, where $I$ is the subspace spanned by the identity
$3\times 3$ matrix, a $\mathfrak{gl}_3(k)$-module structure on $M$ can be
determined by setting
\begin{gather}
(x,a)m=xm+\mu(a)m\qquad \text{for any}\quad (x,a)\in \mathfrak{gl}_3(k),\quad
x\in\mathfrak{sl}_3(k),\quad\text{and}\quad a\in I,\label{eq0}
\end{gather}
where $\mu$ is a linear form on $I$. We denote the obtained $\mathfrak{gl}_3(k)$-module
also by $M$. Using Theorem~\ref{th1} and the isomorphism (see \cite[p.~737]{IT21})
\begin{gather*}
H^n(\mathfrak{gl}_3(k),M)\cong H^n(\mathfrak{sl}_3(k),M)\oplus H^{n-1}(\mathfrak{sl}_3(k),M)
\end{gather*}
for $\mu=0$, we obtain for the cohomology of simple $\mathfrak{gl}_3(k)$-modules the following

\begin{Corollary}\label{c1}
Let $\mathfrak{g}=\mathfrak{gl}_3(k)$, let $M$ be a simple $\mathfrak{g}$-module defined by the formula~\eqref{eq0}. If~$\mu=0$, then the following isomorphisms of $G$-modules hold:

\begin{enumerate}\itemsep=0pt

\item[$(a)$] $H^n(\mathfrak{g})\cong k$ for $n=0,1,3,4,5,6,8,9;$

\item[$(b)$] $H^n(\mathfrak{g},L(p-2,1))\cong
\begin{cases}
L(1,0)^{(1)}&\text{if}\ n=1,2,7,8,
\\
2L(1,0)^{(1)}&\text{if}\ n=4,5;
\end{cases}$

\item[$(c)$] $H^n(\mathfrak{g},L(1,p-2))\cong
\begin{cases}
L(0,1)^{(1)}&\text{if}\ n=1,2,7,8,
\\
2L(0,1)^{(1)}&\text{if}\ n=4,5;
\end{cases}$

\item[$(d)$] $H^n(\mathfrak{g},L(p-3,0))\cong
\begin{cases}
L(1,0)^{(1)}&\text{if}\ n=2,4,5,7,
\\
2L(1,0)^{(1)}&\text{if}\ n=3,6;
\end{cases}$

\item[$(e)$] $H^n(\mathfrak{g},L(0,p-3))\cong
\begin{cases}
L(0,1)^{(1)}&\text{if}\ n=2,4,5,7,
\\
2L(0,1)^{(1)}&\text{if}\ n=3,6;
\end{cases}$

\item[$(f)$] $H^n(\mathfrak{g},L(p-2,p-2))\cong
\begin{cases}
k&\text{if} \ n=1,2,7,8,\\
L(1,1)^{(1)}&\text{if}\ n=3,6,\\
3L(1,1)^{(1)}\oplus 2k&\text{if}\ n=4,5.
\end{cases}$
\end{enumerate}
Otherwise, $H^n(\mathfrak{g},M)=0$.
\end{Corollary}

The results of Corollary~\ref{c1} can be applied to describe the cohomology of the general Lie algebra of Cartan type $W_3(\mathbf{m})$ (for the definition of $W_n(\mathbf{m})$, see \cite[Section~1.3]{IT21}). For example, using in \cite[Theorem~0.2]{BY17} and the statement $(a)$ of Corollary~\ref{c1}, one can easily describe the cohomology of the restricted Lie algebra of Cartan type $W_3(\mathbf{1})$ with coefficients in the divided power algebra.

Let $V(\lambda)$ be the Weyl module with highest weight $\lambda=r\omega_1+s\omega_2$ (for the definition, see Section~\ref{ss22}) and $H^0(\lambda)=V(-w_0(\lambda))^*$, where $w_0$ is the longest element of the Weyl group $W$ of the Lie algebra $\mathfrak{g}$. As an $G$-module, $H^0(\lambda)$ is isomorphic to the induced $G$-module $\operatorname{Ind}_B^{G}(k_{\lambda})$, where $B$ is the Borel subgroup of $G$, corresponding to the negative roots, and $k_{\lambda}$ is a one-dimensional $B$-module. A module $V$ over $G$ is {\it $G$-acyclic}, if $H^n(G,V)=0$ for all $n>0$. The restricted weights $\lambda$ and $\mu$ are {\it linked} if there is $w\in W$ such that
\begin{gather*}
\lambda+\rho\equiv w(\mu+\rho)\mod pX(T),
\end{gather*}
where $\rho$ is the half-sum of positive roots and $X(T)$ is the additive character group of the maximal torus $T$ of $G$. We say that two $G$-modules with highest weights are linked if their highest weights are linked.
As is well-known, $H^0(\lambda)$ and $V(\lambda)$ are $G$-acyclic, see \cite[Corollary~3.4]{CPSK77}. But, from the proof of Theorem~\ref{th1}, we will see that the $\mathfrak{g}$-modules $H^0(\lambda)$ and $V(\lambda)$, linked with simple peculiar modules, are peculiar for $\mathfrak{g}$. For the cohomology of these modules, the following result occurs:

\begin{Corollary}\label{c2} Let $\mathfrak{g}=\mathfrak{sl}_3(k)$ and $V=H^0(\lambda)$. Then, the following isomorphisms of $G$-mo\-dules hold:

\begin{enumerate}\itemsep=0pt

\item[$(a)$] $H^n(\mathfrak{g},H^0(0,0))\cong k$ for $n=0,3,5,8;$

\item[$(b)$] $H^n(\mathfrak{g},H^0(p-2,1))\cong
\begin{cases}
L(1,0)^{(1)}&\text{if}\ n=1,2,3,5,6,7,\\
2L(1,0)^{(1)}&\text{if}\ n=4;
\end{cases}$

\item[$(c)$] $H^n(\mathfrak{g},H^0(1,p-2))\cong
\begin{cases}
L(0,1)^{(1)}&\text{if}\ n=1,2,3,5,6,7,\\
2L(0,1)^{(1)}&\text{if}\ n=4;
\end{cases}$

\item[$(d)$] $H^n(\mathfrak{g},H^0(p-3,0))\cong L(1,0)^{(1)}$ for $n=2,3,5,6;$

\item[$(e)$] $H^n(\mathfrak{g},H^0(0,p-3))\cong L(0,1)^{(1)}$ for $n=2,3,5,6;$

\item[$(f)$] $H^n(\mathfrak{g},H^0(p-2,p-2))\cong
\begin{cases}
L(1,1)^{(1)}&\text{if}\ n=3,\\
2L(1,1)^{(1)}\oplus k&\text{if}\ n=4,\\
L(1,1)^{(1)}\oplus k&\text{if}\ n=5,\\
k&\text{if}\ n=7,8.
\end{cases}$
\end{enumerate}
 Otherwise, $H^n(\mathfrak{g},V)=0$.
\end{Corollary}

\begin{Corollary}\label{c3} Let $\mathfrak{g}=\mathfrak{sl}_3(k)$ and $V=V(\lambda)$. Then, the following isomorphisms of $G$-modules hold:

\begin{enumerate}\itemsep=0pt

\item[$(a)$] $H^n(\mathfrak{g},V(0,0))\cong k$ for $n=0,3,5,8;$

\item[$(b)$] $H^n(\mathfrak{g},V(p-2,1))\cong
\begin{cases}
L(1,0)^{(1)}&\text{if}\ n=1,2,3,5,6,7,
\\
2L(1,0)^{(1)}&\text{if}\ n=4;
\end{cases}$

\item[$(c)$] $H^n(\mathfrak{g},V(1,p-2))\cong
\begin{cases}
L(0,1)^{(1)}&\text{if}\ n=1,2,3,5,6,7,\\
2L(0,1)^{(1)}&\text{if}\ n=4;
\end{cases}$

\item[$(d)$] $H^n(\mathfrak{g},V(p-3,0))\cong L(1,0)^{(1)}$ for $n=2,3,5,6;$

\item[$(e)$] $H^n(\mathfrak{g},V(0,p-3))\cong L(0,1)^{(1)}$ for $n=2,3,5,6;$

\item[$(f)$] $H^n(\mathfrak{g},V(p-2,p-2))\cong
\begin{cases}
k&\text{if}\ n=0,1,\\
L(1,1)^{(1)}\oplus k&\text{if}\ n=3,\\
2L(1,1)^{(1)}\oplus k&\text{if}\ n=4,\\
L(1,1)^{(1)}&\text{if}\ n=5.
\end{cases}$
\end{enumerate}
Otherwise, $H^n(\mathfrak{g},V)=0$.
\end{Corollary}

\section{Preliminary facts}

\subsection{Properties of cohomology}

In this section, we give some properties of the cohomology for the Lie algebra $\mathfrak{g}$ that are used to prove the main results.
Cohomology
\begin{gather*}
H^{\bullet}({\mathfrak{g}},M)=\bigoplus_{n\geq0}H^n({\mathfrak{g}},M)
\end{gather*}
can be computed using a complex $\big(\bigwedge^{\bullet}{\mathfrak{g}}^*\bigotimes M,d\big)$, see \cite[Section~I.9.17]{Jan03}. Therefore, we can identify the space of cochains $C^n({\mathfrak{g}},M)$ with the space $\bigwedge^n{\mathfrak{g}}^*\bigotimes M$ and regard the space $C^n({\mathfrak{g}},M)$ as the $G$-module. We decompose the space of cochains $C^{n}(\mathfrak{g},M)$ into a direct sum
of the eigenspaces with respect to the maximal torus $T$ of the group $G$:
\begin{gather*}
C^{n}({\mathfrak{g}},M)=\bigoplus_{\mu\in
X(T)}C^{n}_{\mu}(\mathfrak{g},M),
\end{gather*}
where $X(T)$ is the additive character group of the torus $T$.
Then,
\begin{gather*}
H^n({\mathfrak{g}},M)=\bigoplus_{\mu\in
X(T)}H^n_{\mu}({\mathfrak{g}},M).
\end{gather*}

Denote by $\prod(V)$
the set of weights of the subspace $V$ of the $G$-module $C^{n}(\mathfrak{g},M)$.
Since
\begin{gather*}
\prod(H^n(\mathfrak{g},M))\subseteq
pX(T)\bigcap\prod\Big(\bigwedge{^n}\mathfrak{g}^*\bigotimes M\Big),
\end{gather*}
we will consider
elements of the subspace $\overline{C}^n(\mathfrak{g},M)$ of the space $C^n(\mathfrak{g},M)$ with weights from the set
\begin{gather*}
pX(T)\bigcap\prod\Big(\bigwedge{^n}\mathfrak{g}^*\bigotimes M\Big).
\end{gather*}
The corresponding subspaces of cocycles
and cohomology are denoted by $\overline{Z}^n\!(\mathfrak{g},M)$ and $\overline{H}^n\!(\mathfrak{g},M)$, respectively. Note that
\begin{gather*}
H^n(\mathfrak{g},M)=\overline{H}^n(\mathfrak{g},M).
\end{gather*}
By the definition of $H^n(\mathfrak{g},M)$,
\begin{gather*}\dim H^n(\mathfrak{g},M)=\dim Z^n(\mathfrak{g},M)-\dim B^{n}(\mathfrak{g},M),\end{gather*} and by the definition of $B^{n}(\mathfrak{g},M)$,
\begin{gather*}\dim B^{n}(\mathfrak{g},M)=\dim C^{n-1}(\mathfrak{g},M)-\dim Z^{n-1}(\mathfrak{g},M).\end{gather*} Then, we get
\begin{gather}
\dim H^n(\mathfrak{g},M)=\dim \overline{Z}^n(\mathfrak{g},M)+\dim
\overline{Z}^{n-1}(\mathfrak{g},M)-\dim \overline{C}^{n-1}(\mathfrak{g},M).\label{dim1}
\end{gather}
Since $\operatorname{Tr}(\operatorname{ad}_x)=0$ for all $x\in \mathfrak{g}$, then, according to the main theorem in \cite[p.~639]{H70},
we get the following isomorphism:
\begin{gather}
H^n(\mathfrak{g},M^*)\cong \big(H^{\dim \mathfrak{g}-n}(\mathfrak{g},M)\big)^*.\label{is1}
\end{gather}
The weight subspaces are invariant under the coboundary operator. Therefore, the formula \eqref{dim1} also holds for weight
subspaces:
\begin{gather}
\dim H^n_{\mu}(\mathfrak{g},M)=\dim \overline{Z}^n_{\mu}(\mathfrak{g},M)+\dim
\overline{Z}^{n-1}_{\mu}(\mathfrak{g},M)-\dim \overline{C}^{n-1}_{\mu}(\mathfrak{g},M).\label{dim3}
\end{gather}

\subsection{Peculiar modules}
\label{ss22}
In this section, we describe simple peculiar $\mathfrak{sl}_3(k)$-modules.
Let $\{e_1,e_2,e_3,h_1,h_2,f_1,f_2,f_3\}$ be the Chevalley basis of $\mathfrak{g}$ with the nonzero brackets
\begin{gather*}
[e_i,f_i]=h_i,\qquad [h_i,e_i]=2e_i,\qquad [h_i,f_i]=-2f_i,\qquad i=1,2,3,
\\
[h_1,e_2]=-e_2,\qquad [h_1,e_3]=e_3,\qquad [h_2,e_1]=-e_1,\qquad [h_2,e_3]=e_3,
\\
[h_1,f_2]=f_2,\qquad [h_1,f_3]=-f_3,\qquad [h_2,f_1]=f_1,\qquad [h_2,f_3]=-f_3,
\\
[e_1,e_2]=e_3,\qquad [e_3,f_1]=-e_2,\qquad [e_3,f_2]=e_1,
\\
[f_1,f_2]=-f_3,\qquad [e_1,f_3]=-f_2,\qquad [e_2,f_3]=f_1,
\end{gather*}
where $h_3=h_1+h_2$.
It is known (see \cite[p.~145]{DI02}) that there are six simple peculiar $\mathfrak{g}$-modules:
\begin{gather*}
L(0,0),\quad L(p-2,1),\quad L(1,p-2),\quad L(p-3,0),\quad L(0,p-3),\quad L(p-2,p-2).
\end{gather*}
A {\it linear span} of a set $\{v_1,\dots,v_m\}$ of vectors of a vector space $V$ over $k$ is the smallest linear subspace of $V$ that contains the set $\{v_1,\dots,v_m\}$. Let $\langle v_1,\dots,v_m\rangle_{k}$ denote the linear span of the set $\{v_1,\dots,v_m\}$ of vectors of the vector space $V$ over $k$. For a detailed description of the peculiar simple modules, consider the restricted Verma module
\begin{gather*}
W(r,s):=\bigg\langle v_{i,j,t}:=\frac{f_3^tf_2^jf_1^i}{t!j!i!}u_{r,s}\,\bigg|\,0\leq i,j,t\leq p-1\bigg\rangle_{k}
\end{gather*}
with the following action of $\mathfrak{sl}_3(k)$:
\begin{gather*}
e_1v_{i,j,t}=-(j+1)v_{i,j+1,t-1}+(r-i+1)v_{i-1,j,t},
\\
e_2v_{i,j,t}=(s+i-j-t+1)v_{i,j-1,t}+(i+1)v_{i+1,j,t-1},
\\
e_3v_{i,j,t}=(r+s-i-j-t+1)v_{i,j,t-1}+(r-i+1)v_{i-1,j-1,t},
\\
h_1v_{i,j,t}=(r-2i+j-t)v_{i,j,t},
\qquad
h_2v_{i,j,t}=(s+i-2j-t)v_{i,j,t},
\\
f_1v_{i,j,t}=-(t+1)v_{i,j-1,t+1}+(i+1)v_{i+1,j,t},
\\
f_2v_{i,j,t}=(j+1)v_{i,j+1,t},
\qquad
f_3v_{i,j,t}=(t+1)v_{i,j,t+1}.
\end{gather*}

The restricted Verma module $W(r,s)$ has a submodule $I(r,s)$ generated by the vectors $v_{r+1,0,0}$ and $v_{0,s+1,0}$.
The quotient $V(r,s)=W(r,s)/I(r,s)$ is also restricted; let us call it the {\it Weyl module}. In the modular case, the term ``Weyl module'' was first used in \cite[p.~321]{S80}, see also \cite[p.~59]{F88}. For groups of Lie type over a field of positive characteristic, the term ``Weyl module'' has also been used for a long time (see, for example, \cite[p.~213]{CaL74}, \cite[p.~262]{Humph80}, \cite[p.~291]{Jan80}). Obviously, for the Weyl module $V(r,s)$, the following relations hold:
\begin{gather*}
v_{r+1,0,0}=0,\qquad v_{0,s+1,0}=0.
\end{gather*}

The submodule structure of $V(r,s)$ is well-known, see \cite[p.~484]{B67}, \cite[pp.~151 and~157]{Rud78}. The results of these papers say that the quotient of $V(r,s)$ by the
maximal submodule is the restricted simple module isomorphic to $L(r,s)$. In~particular, for the peculiar simple $\mathfrak{g}$-modules we get
\begin{gather*}
L(0,0)=V(0,0),\qquad L(p-3,0)=V(p-3,0),\qquad L(0,p-3)=V(0,p-3),
\\
L(p-2,1)=V(p-2,1)/L(p-3,0),\qquad L(1,p-2)=V(1,p-2)/L(0,p-3),
\\
L(p-2,p-2)=V(p-2,p-2)/L(0,0).
\end{gather*}
To describe simple modules, we will use the basis vectors of the corresponding restricted Verma modules. The maximal submodules of these Weyl modules are generated by the highest weight vectors
\begin{gather*}
w_{p-3,0}=v_{1,1,0}-2v_{0,0,1}\qquad \text{for}\quad V(p-2,1),
\\
w_{0,p-3}=v_{1,1,0}+v_{0,0,1}\qquad \text{for}\quad V(1,p-2),
\\
w_{0,0}=\sum_{i=0}^{p-2}(p-2-i)i!v_{p-2-i,p-2-i,i}\qquad \text{for}\quad V(p-2,p-2),
\end{gather*}
respectively.
Then, for simple non-trivial peculiar modules, we obtain the following descriptions in terms of the basis vectors of the restricted Verma module:
\begin{gather*}
L(p-3,0)=\langle v_{i,j,t}\mid 0\leq i\leq p-3,\,0\leq j\leq i,\,0\leq t\leq p-3-i\rangle_{k},
\\
L(0,p-3)=\langle v_{i,j,t}\mid 0\leq i\leq j,\,0\leq j\leq p-3,\,0\leq t\leq p-3-j\rangle_{k},
\\
L(p-2,1)=\langle v_{i,j,t}\mid 0\leq i\leq p-2,\,0\leq j\leq i+1,\,0\leq t\leq p-1-i;\,w_{p-3,0}=0\rangle_{k},
\\
L(1,p-2)=\langle v_{i,j,t}\mid 0\leq i\leq j+1,\,0\leq j\leq p-2,\,0\leq t\leq p-1-j;\,w_{0,p-3}=0\rangle_{k},
\\
L(p-2,p-2)=\langle v_{i,j,t}\mid 0\leq i\leq p-2,\,0\leq j\leq p-2,\,0\leq t\leq p-2-i-j;\,w_{0,0}=0\rangle_{k}.
\end{gather*}

\section{Proof of Theorem~\ref{th1}}

As noted above, there are only six peculiar simple modules. Let us prove the theorem for each peculiar simple module separately.

 $(a)$ Since $p>3$, then the Killing form on $\mathfrak{g}$ is non-degenerate. Then, for the trivial one-dimensional module $M=L(0,0)$, the result obtained earlier for zero characteristic remains true in our case as well. So, we consider only non-trivial peculiar simple modules.

$(b)$ Let $M=L(p-2,1)$.

\begin{Lemma}\label{l1}
Let $\mathfrak{g}=\mathfrak{sl}_3(k)$ and $M=L(p-2,1)$. Then, $H^n(\mathfrak{g},M)=0$, except for the following cases:
\begin{enumerate}\itemsep=0pt
\item[$(i)$] $H^1(\mathfrak{g},L(p-2,1))\cong H^7(\mathfrak{g},L(p-2,1))\cong L(1,0)^{(1)}$,
\item[$(ii)$] $H^4(\mathfrak{g},L(p-2,1))\cong 2L(1,0)^{(1)}$.
\end{enumerate}
\end{Lemma}

\begin{proof}
Obviously, $H^0(\mathfrak{g},L(p-2,1))=0$. It is also known that $H^1(\mathfrak{g},L(p-2,1))\cong L(1,0)^{(1)}$, see \cite[p.~301]{Jan91}, and $H^2(\mathfrak{g},L(p-2,1))=0$, see \cite[Theorem~1.1]{DI02}.

Now we show that $H^3(\mathfrak{g},L(p-2,1))=0$. We get
\begin{gather*}
\prod\big(\overline{C}^{\bullet}(\mathfrak{g},L(p-2,1))\big)=\{p\omega_1, p(-\omega_1+\omega_2), -p\omega_2\}.
\end{gather*}
The subspace $\overline{C}^2(\mathfrak{g},L(p-2,1))$ is $21$-dimensional and its set of weights consists of three ele\-ments~$p\omega_1$, $p(-\omega_1+\omega_2)$, $-p\omega_2$. We have
\begin{align*}
\dim \overline{C}^2_{p\omega_1}(\mathfrak{g},L(p-2,1))
&=\dim \overline{C}^2_{p(-\omega_1+\omega_2)}(\mathfrak{g},L(p-2,1))
\\
&=\dim \overline{C}^2_{-p\omega_2}(\mathfrak{g},L(p-2,1))=7.
\end{align*}

The subspace $\overline{C}^2_{p\omega_1}(\mathfrak{g},L(p-2,1))$ is spanned by the $2$-cochains
\begin{gather*}
\psi_1^2=h_1^*\wedge f_1^*\otimes v_{0,0,0},\qquad
\psi_2^2=h_2^*\wedge f_1^*\otimes v_{0,0,0},\qquad
\psi_3^2=e_2^*\wedge f_3^*\otimes v_{0,0,0},
\\
\psi_4^2=h_1^*\wedge f_3^*\otimes v_{0,1,0},\qquad
\psi_5^2=h_2^*\wedge f_3^*\otimes v_{0,1,0},
\\
\psi_6^2=f_1^*\wedge f_2^*\otimes v_{0,1,0},\qquad
\psi_7^2=f_1^*\wedge f_3^*\otimes v_{0,0,1}.
\end{gather*}
Let $\sum_{i=1}^7a_i\psi_i^2\in Z^2(\mathfrak{g},L(p-2,1))$, where $a_i\in k$ for all $i$. Then, by the cocycle condition,
\begin{gather*}
a_1=a_2=a_4=a_5=0,\qquad
a_3=a_6=a_7.
\end{gather*}
This means that $\dim Z^2_{p\omega_1}(\mathfrak{g},L(p-2,1))=1$.

The subspace $\overline{C}^3_{p\omega_1}(\mathfrak{g},L(p-2,1))$ is spanned by the $3$-cochains
\begin{alignat*}{3}
&\psi_1^3=h_1^*\wedge h_2^*\wedge f_1^*\otimes v_{0,0,0},\qquad&&
\psi_2^3=h_1^*\wedge e_2^*\wedge f_3^*\otimes v_{0,0,0},&
\\
&\psi_3^3=h_2^*\wedge e_2^*\wedge f_3^*\otimes v_{0,0,0},\qquad&&
\psi_4^3=e_2^*\wedge f_1^*\wedge f_1^*\otimes v_{0,0,0},&
\\
&\psi_5^3=e_3^*\wedge f_1^*\wedge f_3^*\otimes v_{0,0,0},\qquad&&
\psi_6^3=e_2^*\wedge f_1^*\wedge f_3^*\otimes v_{1,0,0},&
\\
&\psi_7^3=e_1^*\wedge f_1^*\wedge f_3^*\otimes v_{0,1,0},\qquad&&
\psi_8^3=h_1^*\wedge h_2^*\wedge f_3^*\otimes v_{0,1,0},&
\\
&\psi_9^3=e_2^*\wedge f_2^*\wedge f_3^*\otimes v_{0,1,0},\qquad&&
\psi_{10}^3=h_2^*\wedge f_1^*\wedge f_2^*\otimes v_{0,1,0},&
\\
&\psi_{11}^3=h_1^*\wedge f_1^*\wedge f_2^*\otimes v_{0,1,0},\qquad&&
\psi_{12}^3=h_1^*\wedge f_1^*\wedge f_3^*\otimes v_{0,0,1},&
\\
&\psi_{13}^3=h_2^*\wedge f_1^*\wedge f_3^*\otimes v_{0,0,1},\qquad&&
\psi_{14}^3=f_1^*\wedge f_2^*\wedge f_3^*\otimes v_{0,1,1}.&
\end{alignat*}
So, $\dim \overline{C}^3_{p\omega_1}(\mathfrak{g},L(p-2,1))=14$. Suppose $\sum_{i=1}^{14}b_i\psi_i^3\in Z^3(\mathfrak{g},L(p-2,1))$, where $b_i\in k$ for all~$i$. Then, using the cocycle condition, we get
\begin{gather*}
b_1=b_8=0,
\qquad
b_2+b_4-b_5-2b_6-b_7=0,
\qquad
-b_2-b_3+b_5+b_7+b_9-2b_{14}=0,
\\
b_{10}-b_3=0,
\qquad
b_{11}-b_2=0,
\qquad
b_{12}-b_2=0,
\qquad
b_{13}-b_3=0.
\end{gather*}
Consider these equalities as a system of equations for $b_i$, where $i=1,\dots,14$. The rank of the matrix of this system is equal to $8$. Therefore, $\dim Z^3_{p\omega_1}(\mathfrak{g},L(p-2,1))=14-8=6$. Then, by~\eqref{dim3},
\begin{align*}
\dim H^3_{p\omega_1}(\mathfrak{g},L(p-2,1))={}&
\dim Z^3_{p\omega_1}(\mathfrak{g},L(p-2,1))+\dim Z^2_{p\omega_1}(\mathfrak{g},L(p-2,1))
\\
&-\dim C^2_{p\omega_1}(\mathfrak{g},L(p-2,1))=6+1-7=0.
\end{align*}
Thus, all $3$-cocycles with dominant highest weight $p\omega_1$ are coboundaries. Therefore,
\begin{gather*}
H^3(\mathfrak{g},L(p-2,1))=0.
\end{gather*}

Now we will calculate $H^4(\mathfrak{g},L(p-2,1))$. The subspace $\overline{C}^4_{p\omega_1}(\mathfrak{g},L(p-2,1))$ is $18$-dimensional and is spanned by the $4$-cochains
\begin{alignat*}{3}
&\psi_1^4=h_1^*\wedge h_2^*\wedge e_2^*\wedge f_3^*\otimes v_{0,0,0},\qquad&&
\psi_2^4=h_1^*\wedge e_2^*\wedge f_1^*\wedge f_2^*\otimes v_{0,0,0},&
\\
&\psi_3^4=h_2^*\wedge e_2^*\wedge f_1^*\wedge f_2^*\otimes v_{0,0,0},\qquad&&
\psi_4^4=h_1^*\wedge e_3^*\wedge f_1^*\wedge f_3^*\otimes v_{0,0,0},&
\\
&\psi_5^4=h_2^*\wedge e_3^*\wedge f_1^*\wedge f_3^*\otimes v_{0,0,0},\qquad&&
\psi_6^4=e_1^*\wedge e_2^*\wedge f_1^*\wedge f_3^*\otimes v_{0,0,0},&
\\
&\psi_7^4=h_1^*\wedge e_2^*\wedge f_1^*\wedge f_3^*\otimes v_{1,0,0},\qquad&&
\psi_8^4=h_2^*\wedge e_2^*\wedge f_1^*\wedge f_3^*\otimes v_{1,0,0},&
\\
&\psi_9^4=h_1^*\wedge h_2^*\wedge f_1^*\wedge f_2^*\otimes v_{1,0,0},\qquad&&
\psi_{10}^4=h_1^*\wedge e_1^*\wedge f_1^*\wedge f_3^*\otimes v_{0,1,0},&
\\
&\psi_{11}^4=h_2^*\wedge e_1^*\wedge f_1^*\wedge f_3^*\otimes v_{0,1,0},\qquad&&
\psi_{12}^4=h_1^*\wedge e_2^*\wedge f_2^*\wedge f_3^*\otimes v_{0,1,0},&
\\
&\psi_{13}^4=h_2^*\wedge e_2^*\wedge f_2^*\wedge f_3^*\otimes v_{0,1,0},\qquad&&
\psi_{14}^4=e_3^*\wedge f_1^*\wedge f_2^*\wedge f_3^*\otimes v_{0,1,0},&
\\
&\psi_{15}^4=h_1^*\wedge h_2^*\wedge f_1^*\wedge f_3^*\otimes v_{0,0,1},\qquad&&
\psi_{16}^4=e_2^*\wedge f_1^*\wedge f_2^*\wedge f_3^*\otimes v_{0,0,1},&
\\
&\psi_{17}^4=h_1^*\wedge f_1^*\wedge f_2^*\wedge f_3^*\otimes v_{0,1,1},\qquad&&
\psi_{18}^4=h_2^*\wedge f_1^*\wedge f_2^*\wedge f_3^*\otimes v_{0,1,1}.&
\end{alignat*}
Suppose $\sum_{i=1}^{18}c_i\psi_i^4\in Z^4(\mathfrak{g},L(p-2,1))$, where $c_i\in k$ for all $i$. Then, using the cocycle condition, we get
\begin{gather*}
c_1=c_9=c_{15},
\qquad
c_1-c_3+c_5+2c_8+c_{11}=0,
\qquad
c_2+c_3+c_5-c_6-c_{14}-c_{16}=0,
\\
c_2-2c_7+c_{12}-c_{15}-2c_{17}=0,
\qquad
c_3-2c_8+c_{13}-2c_{18}=0,
\\
c_4-c_2+2c_7+c_{10}=0,
\qquad
c_4-c_9+c_{10}+c_{12}-2c_{17}=0,
\\
c_5+c_9+c_{11}+c_{13}-2c_{18}=0,
\qquad
c_6-c_{11}-c_{12}+c_{14}+c_{16}=0.
\end{gather*}
The rank of the matrix of this system is equal to $8$. Therefore,
\begin{gather*}
\dim Z^4_{p\omega_1}(\mathfrak{g},L(p-2,1))=18-8=10.
\end{gather*}
Then, by \eqref{dim3},
\begin{align*}
\dim H^4_{p\omega_1}(\mathfrak{g},L(p-2,1))={}&
\dim Z^4_{p\omega_1}(\mathfrak{g},L(p-2,1))+\dim Z^3_{p\omega_1}(\mathfrak{g},L(p-2,1))
\\
&-\dim C^3_{p\omega_1}(\mathfrak{g},L(p-2,1))=10+6-14=2.
\end{align*}
Thus, $H^4(\mathfrak{g},L(p-2,1))$ is generated by the two cohomological classes with weight $p\omega_1$. So, $H^4(\mathfrak{g},L(p-2,1))\cong 2L(1,0)^{(1)}$.

Similar calculations give us
\begin{alignat*}{3}
&\dim C^5_{p\omega_1}(\mathfrak{g},L(p-2,1))=14, \qquad&&\dim Z^5_{p\omega_1}(\mathfrak{g},L(p-2,1))=8,&
\\
&\dim C^6_{p\omega_1}(\mathfrak{g},L(p-2,1))=7, \qquad&&\dim Z^6_{p\omega_1}(\mathfrak{g},L(p-2,1))=6,&
\\
&\dim C^7_{p\omega_1}(\mathfrak{g},L(p-2,1))=2, \qquad&&\dim Z^7_{p\omega_1}(\mathfrak{g},L(p-2,1))=2.&
\end{alignat*}
Then, using \eqref{dim3}, we get
\begin{gather*}
\dim H^5_{p\omega_1}(\mathfrak{g},L(p-2,1))=0,\qquad
\dim H^6_{p\omega_1}(\mathfrak{g},L(p-2,1))=0,
\\
\dim H^7_{p\omega_1}(\mathfrak{g},L(p-2,1))=1.
\end{gather*}
This completes the proof of the lemma.
\end{proof}

$(c)$ Let $M=L(1,p-2)$. Obviously, $M$ is dual to $L(p-2,1)$. Then, using \eqref{is1} and Lemma~\ref{l1}, we get the following

\begin{Lemma}
Let $\mathfrak{g}=\mathfrak{sl}_3(k)$ and $M=L(1,p-2)$. Then, $H^n(\mathfrak{g},M)=0$, except for the following cases:

\begin{enumerate}\itemsep=0pt

\item[$(i)$] $H^1(\mathfrak{g},L(1,p-2))\cong H^7(\mathfrak{g},L(1,p-2))\cong L(0,1)^{(1)}$,

\item[$(ii)$] $H^4(\mathfrak{g},L(1,p-2))\cong 2L(0,1)^{(1)}$.

\end{enumerate}
\end{Lemma}

$(d)$ Let $M=L(p-3,0)$.

\begin{Lemma}\label{l3}
Let $\mathfrak{g}=\mathfrak{sl}_3(k)$ and $M=L(p-3,0)$. Then, $H^n(\mathfrak{g},M)=0$, except for the following cases:
\begin{enumerate}\itemsep=0pt

\item[$(i)$] $H^2(\mathfrak{g},L(p-3,0))\cong H^6(\mathfrak{g},L(p-3,0))\cong L(1,0)^{(1)}$,

\item[$(ii)$] $H^3(\mathfrak{g},L(p-3,0))\cong H^5(\mathfrak{g},L(p-3,0))\cong L(1,0)^{(1)}$.
\end{enumerate}
\end{Lemma}

\begin{proof}
It is easy to see that
\begin{gather*}
\prod\big(\overline{C}^{\bullet}(\mathfrak{g},L(p-3,0))\big)=
\{p\omega_1, p(-\omega_1+\omega_2), -p\omega_2\}.
\end{gather*}
Then, it is obvious that
\begin{gather*}
\prod\big(\overline{C}^i(\mathfrak{g},L(p-3,0))\big)\bigcap \prod\big(\overline{C}^*(\mathfrak{g},L(p-3,0))\big)=\varnothing
\qquad\text{for}\quad i=1,2.
\end{gather*}
Therefore,
\begin{gather*}
H^0(\mathfrak{g},L(p-3,0))=0\qquad\text{and}\qquad H^1(\mathfrak{g},L(p-3,0))=0.
\end{gather*}

Further, we get
\begin{gather*}
\prod\big(\overline{C}^2(\mathfrak{g},L(p-3,0))\big)=\{p\omega_1, p(-\omega_1+\omega_2), -p\omega_2\}.
\end{gather*}
 Any composition factor of $H^2(\mathfrak{g},L(p-3,0))$, as a $G$-module, is uniquely determined by its highest weight. The highest weight of a simple $G$-module is dominant, see~\cite[p.~260]{Humph80}. Recall that the weight $\lambda=r\omega_1+s\omega_2$ is dominant if $r\geq 0$ and $s\geq 0$. Then, $H^2(\mathfrak{g},L(p-3,0))$ can be generated only by the classes of cocycles with dominant weight $p\omega_1$. Therefore, it is sufficient to determine the multiplicity of $p\omega_1$. Consider the subspace of $2$-cochains with dominant weight~$p\omega_1$. The subspace
$\overline{C}^2_{p\omega_1}(\mathfrak{g},L(p-3,0))$ is one-dimensional and is spanned by the 2-cochain $\psi^2=f_1^*\wedge f_3^*\otimes v_{0,0,0}$. It is easy to see that $\psi^2$ is a 2-cocycle. Since $\overline{C}^1_{p\omega_1}(\mathfrak{g},L(p-3,0))\allowbreak=0$, it follows that $\psi^2$ cannot be a coboundary. Therefore, $H^2(\mathfrak{g},L(p-3,0))$, as a $G$-module, is generated by the class $\big[\psi^2\big]$ of $2$-cocycles with weight $p\omega_1$ and is isomorphic to $L(1,0)^{(1)}$.

The set of weights of the subspace $\prod\big(\overline{C}^3(\mathfrak{g},L(p-3,0))\big)$ is also equal to $\{p\omega_1,p(-\omega_1+\omega_2),\allowbreak -p\omega_2\}$. The subspace
$\overline{C}^3_{p\omega_1}(\mathfrak{g},L(p-3,0))$ is two-dimensional and is spanned by the $3$-cochains
\begin{gather*}
\psi^3_1=h_1^*\wedge f_1^*\wedge f_3^*\otimes v_{0,0,0},\qquad\psi^3_2=h_2^*\wedge f_1^*\wedge f_3^*\otimes v_{0,0,0}.
\end{gather*}
If $a_1\psi^3_1+a_2\psi^3_2$ is a $3$-cocycle, then it follows from the cocycle condition that $a_2=0$. Since
\begin{gather*}
\dim \overline{C}^2_{p\omega_1}(\mathfrak{g},L(p-3,0))=\dim \overline{Z}^2_{p\omega_1}(\mathfrak{g},L(p-3,0))=1,
\end{gather*}
by~\eqref{dim3}, we see that $\dim H^3_{p\omega_1}(\mathfrak{g},L(p-3,0))=1+1-1=1$. Therefore, $H^3(\mathfrak{g},L(p-3,0))$, as a $G$-module, is generated by the class $\big[\psi^3_1\big]$ of $3$-cocycles with weight $p\omega_1$ and is isomorphic to $L(1,0)^{(1)}$.

 Now, we will prove that $H^4(\mathfrak{g},L(p-3,0))=0$. The weight subspace
$\overline{C}^4_{p\omega_1}(\mathfrak{g},L(p-3,0))$ is two-dimensional and is spanned by the $4$-cochains
\begin{gather*}
\psi^4_1=h_1^*\wedge h_2^*\wedge f_1^*\wedge f_3^*\otimes v_{0,0,0},\qquad\psi^4_2=e_2^*\wedge f_1^*\wedge f_2^*\wedge f_3^*\otimes v_{0,0,0}.
\end{gather*}
If $b_1\psi^4_1+b_2\psi^4_2$
is a $4$-cocycle, then it follows from the cocycle condition that $b_1=0$. Since
\begin{gather*}
\dim \overline{C}^3_{p\omega_1}(\mathfrak{g},L(p-3,0))=2\qquad\text{and}\qquad
\dim \overline{Z}^3_{p\omega_1}(\mathfrak{g},L(p-3,0))=1,
\end{gather*}
by~\eqref{dim3}, it follows that
\begin{gather*}
\dim H^4_{p\omega_1}(\mathfrak{g},L(p-3,0))=1+1-2=0.
\end{gather*}
Therefore,
\begin{gather*}
\dim H^4(\mathfrak{g},L(p-3,0))=\dim H^4_{p\omega_1}(\mathfrak{g},L(p-3,0))=0.
\end{gather*}

The weight subspace
$\overline{C}^5_{p\omega_1}(\mathfrak{g},L(p-3,0))$ is two-dimensional and is spanned by the $5$-cochains
\begin{gather*}
\psi^5_1=h_1^*\wedge e_2^*\wedge f_1^*\wedge f_2^*\wedge f_3^*\otimes v_{0,0,0},\qquad\psi^5_2=h_2^*\wedge e_2^*\wedge f_1^*\wedge f_2^*\wedge f_3^*\otimes v_{0,0,0}.
\end{gather*}
It follows from the cocycle condition that $c_1\psi^5_1+c_2\psi^5_2$ is a $5$-cocycle for any $c_1,c_2\in k$. So, $\dim \overline{Z}^5_{p\omega_1}(\mathfrak{g},L(p-3,0))=2$. Since
\begin{gather*}
\dim \overline{C}^4_{p\omega_1}(\mathfrak{g},L(p-3,0))=2\qquad\text{and}\qquad
\dim \overline{Z}^4_{p\omega_1}(\mathfrak{g},L(p-3,0))=1,
\end{gather*}
then by \eqref{dim3}, we see that
\begin{gather*}
\dim H^5_{p\omega_1}(\mathfrak{g},L(p-3,0))=2+1-2=1.
\end{gather*}
Therefore, $H^5(\mathfrak{g},L(p-3,0))$, as a $G$-module, is generated by the class $\big[\psi^5_1\big]$ of $5$-cocycles with weight $p\omega_1$ and is isomorphic to $L(1,0)^{(1)}$.

The weight subspace
$\overline{C}^6_{p\omega_1}(\mathfrak{g},L(p-3,0))$ is one-dimensional and is spanned by the $6$-cochain
\begin{gather*}
\psi^6=h_1^*\wedge h_2^*\wedge e_2^*\wedge f_1^*\wedge f_2^*\wedge f_3^*\otimes v_{0,0,0}.
\end{gather*}
It follows from the cocycle condition that $\psi^6$ is a $6$-cocycle. So, $\dim \overline{Z}^6_{p\omega_1}(\mathfrak{g},L(p-3,0))=1$. Since
\begin{gather*}
\dim \overline{C}^5_{p\omega_1}(\mathfrak{g},L(p-3,0))=2\qquad\text{and}\qquad
\dim \overline{Z}^5_{p\omega_1}(\mathfrak{g},L(p-3,0))=2,
\end{gather*}
then by~\eqref{dim3},
\begin{gather*}
\dim H^6_{p\omega_1}(\mathfrak{g},L(p-3,0))=1.
\end{gather*}
Therefore, $H^6(\mathfrak{g},L(p-3,0))$, as a $G$-module, is generated by the class $[\psi^6]$ of $6$-cocycles with weight $p\omega_1$ and is isomorphic to $L(1,0)^{(1)}$.

Finally, the subspaces $\overline{C}^7(\mathfrak{g},L(p-3,0))$ and $\overline{C}^8(\mathfrak{g},L(p-3,0))$ are trivial, therefore,
\begin{equation*}
H^7(\mathfrak{g},L(p-3,0))=0\qquad\text{and}\qquad H^8(\mathfrak{g},L(p-3,0))=0.
\end{equation*}
\renewcommand{\qed}{}
\end{proof}

$(e)$ Let $M=L(0,p-3)$. Obviously, $M$ is dual to $L(p-3,0)$. Then, using \eqref{is1} and Lemma~\ref{l3}, we get the following

\begin{Lemma}\label{l4}
Let $\mathfrak{g}=\mathfrak{sl}_3(k)$ and $M=L(0,p-3)$. Then, $H^n(\mathfrak{g},M)=0$, except for the following cases:
\begin{enumerate}\itemsep=0pt

\item[$(i)$] $H^2(\mathfrak{g},L(0,p-3))\cong H^6(\mathfrak{g},L(0,p-3))\cong L(0,1)^{(1)}$,

\item[$(ii)$] $H^3(\mathfrak{g},L(0,p-3))\cong H^5(\mathfrak{g},L(0,p-3))\cong L(0,1)^{(1)}$.
\end{enumerate}
\end{Lemma}

$(f)$ Finally, let $M=L(p-2,p-2)$. In~this case, we will use some properties of the connection between ordinary and restricted cohomologies. The restricted cohomolology of a restricted Lie algebra with coefficients in a restricted module was introduced by Hochschild in \cite[p.~561]{Hoch54}. The restricted $n$-cohomology of $\mathfrak{g}$ with coefficients in a restricted $\mathfrak{g}$-module $V$ is denoted by~$H^n_{\rm res}(\mathfrak{g},V)$.

For $M=L(p-2,p-2)$, there is the following short exact sequence of $\mathfrak{g}$-modules:
\begin{gather}
0\longrightarrow M\longrightarrow H^0(p-2,p-2)\longrightarrow k\longrightarrow 0.
\label{eq1}
\end{gather}
If the cohomology of $H^0(p-2,p-2)$ is known, then using the long exact cohomology sequence
\begin{gather}
\cdots\longrightarrow H^n(\mathfrak{g},M)\longrightarrow H^n\big(\mathfrak{g},H^0(p-2,p-2)\big)\longrightarrow H^n(\mathfrak{g})\longrightarrow \cdots,\label{eq12}
\end{gather}
corresponding to the short exact sequence~\eqref{eq1}, we can obtain information about $H^{n}(\mathfrak{g},M)$.
We calculate $H^n\big(\mathfrak{g},H^0(p-2,p-2)\big)$ in two steps.

First, we calculate the restricted cohomology
$H^n_{\rm res}\big(\mathfrak{g},H^0(p-2,p-2)\big)$, using the equivalence of the cohomologies $H^n_{\rm res}\big(\mathfrak{g},H^0(p-2,p-2)\big)$ and $H^n\big(G_1,H^0(p-2,p-2)\big)$, where $G_1$ is the first Frobenius kernel for $G$, see~\cite[Section~I.9.6]{Jan03}, and Andersen--Jantzen formula on cohomology of~$G_1$ with coefficients in $H^0(\lambda)$, see~\cite{AJ84}. Let $p>3$, and $\lambda=w\cdot 0+p\nu$. Then, see \cite[p.~501]{AJ84},
\begin{gather}
H^i\big(G_1,H^0(\lambda)\big)^{(-1)}\cong
\begin{cases}
\operatorname{Ind}_B^{G}\big(S^{(i-l(w))/2}(\mathfrak{u}^*)\otimes k_{\nu}\big)&\text{if} \ i-l(w)\ \text{is even},
\\
0&\text{if} \ i-l(w)\ \text{is odd},
\end{cases}
\label{eq13}
\end{gather}
where $\mathfrak{u}$ is the maximal nilpotent subalgebra of $\mathfrak{g}$, corresponding to the negative roots. The Lie algebra $\mathfrak{u}$ is the Lie algebra of the unipotent radical $U$ of $B$.

Then, to pass to the usual cohomology $H^n\big(\mathfrak{g},H^0(p-2,p-2)\big)$, we use the Friedlander--Parshall--Farnsteiner spectral sequence, see \cite[Section~5]{FP88} and \cite[Theorem~4.1]{F91}. In~\cite[Theo\-rem~4.1]{F91}, choosing the zero ideal as an ideal of a given Lie algebra, we obtain the following spectral sequence for the cohomology of the Lie algebra $\mathfrak{g}$ with coefficients in the $\mathfrak{g}$-module $V$:
\begin{gather*}
\bigoplus_{i+j=n}\operatorname{Hom}_k\Big(\bigwedge{^i}(\mathfrak{g}),H^j_{\rm res}(\mathfrak{g},V)\Big)\Longrightarrow H^n(\mathfrak{g},V).
\end{gather*}
In particular, the following lemma can be directly obtained from the last spectral sequence:

\begin{Lemma}\label{l5}
Let $\mathfrak{g}=\mathfrak{sl}_3(k)$
$p>3$ and $V$ a $\mathfrak{g}$-module. Then
\begin{enumerate}\itemsep=0pt

\item[$(i)$] if
$H^{i}_{\rm res}(\mathfrak{g},V)=0$ for all $i\leq n$, then $H^{i}(\mathfrak{g},V)=0$ for all $i\leq n$,

\item[$(ii)$] if $H^i(\mathfrak{g},V)=0$ for all $i\leq n-2$, then
\begin{gather}
H^{n-1}(\mathfrak{g},V)\cong H^{n-1}_{\rm res}(\mathfrak{g},V) \label{eq14}
\end{gather}
\end{enumerate}
and the following sequence is exact:
\begin{align}
0&\longrightarrow H^{n}_{\rm res}(\mathfrak{g},V)\longrightarrow H^{n}(\mathfrak{g},V)\longrightarrow \operatorname{Hom}_k\big(\mathfrak{g},H^{n-1}_{\rm res}(\mathfrak{g},V)\big)\longrightarrow H^{n+1}_{\rm res}(\mathfrak{g},V)\nonumber
\\
&\longrightarrow H^{n+1}(\mathfrak{g},V). \label{eq15}
\end{align}
\end{Lemma}

We start by calculating the cohomology $H^n\big(G_1,H^0(p-2,p-2)\big)$ with $n\leq \dim \mathfrak{g}$.

{\samepage\begin{Lemma}\label{l6}
Let $G_1$ be the first Frobenius kernel of $G$, and $V=H^0(p-2,p-2)$ the $G_1$-module. Then,
\begin{enumerate}\itemsep=0pt

\item[$(i)$] $H^{i}(G_1,V)=0$ for $i=0,1,2,4,6,8$,

\item[$(ii)$] $H^{3}(G_1,V)\cong L(1,1)^{(1)}$,

\item[$(iii)$] $H^5(G_1,V)\cong L(3,0)^{(1)}\oplus L(0,3)^{(1)}\oplus L(2,2)^{(1)}$.
\end{enumerate}
\end{Lemma}

}

\begin{proof}
$(i)$ Since
\begin{gather*}
\lambda=(p-2)(\omega_1+\omega_2)=s_1s_2s_1\cdot 0+p(\omega_1+\omega_2),
\end{gather*}
we get
\begin{gather*}
w=s_1s_2s_1,\qquad l(w)=3,\qquad\text{and}\qquad\nu=\omega_1+\omega_2.
\end{gather*}
Then, by \eqref{eq13},
$H^{i}(G_1,V)=0$ for $i=0,1,2,4,6,8$. The statement $(i)$ is proved.

$(ii)$ We have
\begin{gather*}
S^{(3-l(w))/2}(\mathfrak{u}^*)\otimes k_{\nu}=S^0(\mathfrak{u}^*)\otimes k_{\nu}\cong k_{\nu}=k_{\omega_1+\omega_2}
\end{gather*}
and
\begin{gather*}
H^0\big(S^{(3-l(w))/2}(\mathfrak{u}^*)\otimes k_{\nu}\big)\cong H^0(k_{\omega_1+\omega_2})=\operatorname{Ind}_B^{G}(k_{\omega_1+\omega_2})\cong L(1,1).
\end{gather*}
Then, by \eqref{eq13},
\begin{gather*}
H^3\big(G_1,H^0(p-2,p-2)\big)^{(-1)}\cong L(1,1).
\end{gather*}

$(iii)$ We have
\begin{gather*}
S^{(5-l(w))/2}(\mathfrak{u})^*\otimes k_{\nu}\cong (k_{\alpha_1}\oplus
k_{\alpha_2}\oplus k_{\alpha_1+\alpha_2})\otimes k_{\omega_1+\omega_2}\cong
k_{3\omega_1}\oplus k_{3\omega_2}\oplus k_{2\omega_1+2\omega_2}
\end{gather*}
and
\begin{gather*}
H^0\big(S^{(5-l(w))/2}(\mathfrak{u}^*)\otimes k_{\nu}\big)=\operatorname{Ind}_B^{G}(k_{3\omega_1}\oplus k_{3\omega_2}\oplus k_{2\omega_1+2\omega_2})
 \cong L(3,0)\oplus L(0,3)\oplus L(2,2).
 \end{gather*}
Then, by \eqref{eq13},
\begin{equation*}
H^5\big(G_1,H^0(p-2,p-2)\big)^{(-1)}\cong L(3,0)\oplus L(0,3)\oplus L(2,2).
\end{equation*}
\renewcommand{\qed}{}
\end{proof}

Now, we calculate the cohomology $H^n\big(\mathfrak{g},H^0(p-2,p-2)\big)$.

\begin{Lemma}\label{l7}
Let $\mathfrak{g}=\mathfrak{sl}_3(k)$ and $V=H^0(p-2,p-2)$. Then,
\begin{enumerate}\itemsep=0pt
\item[$(i)$] $H^{i}(\mathfrak{g},V)=0 $ for $i=0,1,2$,
\item[$(ii)$] $H^{3}(\mathfrak{g},V)\cong L(1,1)^{(1)}$,
\item[$(iii)$] $H^4(\mathfrak{g},V)\cong 2L(1,1)^{(1)}\oplus k$.
\end{enumerate}
\end{Lemma}

\begin{proof}
$(i)$ Follows from the statements $(i)$ of Lemmas~\ref{l5} and \ref{l6}.
$(ii)$ Follows from the statements $(ii)$ of Lemma~\ref{l6} and formula \eqref{eq14}.
$(iii)$ We get
\begin{align*}
\operatorname{Hom}_k\big(\mathfrak{g},H^3_{\rm res}\big(\mathfrak{g},H^0(p-2,p-2)\big)^{(-1)}\big)&\cong \mathfrak{g}^*\otimes H^3\big(G_1,H^0(p-2,p-2)\big)^{(-1)}
\\
&\cong L(1,1)\otimes L(1,1)
\\
&\cong
L(3,0)\oplus L(0,3)\oplus L(2,2)\oplus 2L(1,1)\oplus k.
\end{align*}
Then, by \eqref{eq15} and the statement $(iii)$ of Lemma~\ref{l6},
\begin{equation*} H^4\big(\mathfrak{g},H^0(p-2,p-2)\big)^{(-1)}\cong 2L(1,1)\oplus k.
\end{equation*}
\renewcommand{\qed}{}
\end{proof}

For $H^n(\mathfrak{g},M)$, where $M=L(p-2,p-2)$, we obtain the following result:

\begin{Lemma}\label{l8}
Let $\mathfrak{g}=\mathfrak{sl}_3(k)$ and $M=L(p-2,p-2)$. Then, $H^n(\mathfrak{g},M)=0$, except for the following cases:

\begin{enumerate}\itemsep=0pt

\item[$(i)$] $H^{1}(\mathfrak{g},M)\cong H^{7}(\mathfrak{g},M)\cong k$,

\item[$(ii)$] $H^{3}(\mathfrak{g},M)\cong H^{5}(\mathfrak{g},M)\cong L(1,1)^{(1)}$,

\item[$(iii)$] $H^4(\mathfrak{g},M)\cong 2L(1,1)^{(1)}\oplus 2k$.

\end{enumerate}
\end{Lemma}

\begin{proof} Obviously, $H^0\big(\mathfrak{g},H^0(p-2,p-2)\big)=H^8\big(\mathfrak{g},H^0(p-2,p-2)\big)=0$.

$(i)$ The initial terms of the exact sequence~\eqref{eq12} give us the following exact sequence:
\begin{gather*}0\longrightarrow H^0(\mathfrak{g})\longrightarrow H^1(\mathfrak{g},L(p-2,p-2))\longrightarrow H^1\big(\mathfrak{g},H^0(p-2,p-2)\big).\end{gather*} By Lemma~\ref{l7},
\begin{gather*}H^1(\mathfrak{g},L(p-2,p-2))\cong H^0(\mathfrak{g})\cong k.\end{gather*}
Since $L(p-2,p-2)$ is a self-dual module, then by~\eqref{is1}, we get
\begin{gather*}H^7(\mathfrak{g},L(p-2,p-2))\cong H^1(\mathfrak{g},L(p-2,p-2))^*\cong k.
\end{gather*}

Let us prove that $H^2(\mathfrak{g},L(p-2,p-2))=0$. Obviously, $H^1(\mathfrak{g})=0$, and by Lemma~\ref{l7}, $H^2\big(\mathfrak{g},H^0(p-2,p-2)\big)=0$. Then, it follows from the exactness of the sequence~\eqref{eq12} that $H^2(\mathfrak{g},L(p-2,p-2))=0$. By~\eqref{is1},
\begin{gather*}H^6(\mathfrak{g},L(p-2,p-2))\cong H^2(\mathfrak{g},L(p-2,p-2))^*=0.\end{gather*}

$(ii)$ Since $H^2(\mathfrak{g})=0$, it follows from the exactness of the sequence~\eqref{eq12} that the sequence
\begin{gather} 0\longrightarrow H^3(\mathfrak{g},L(p-2,p-2))\longrightarrow H^3\big(\mathfrak{g},H^0(p-2,p-2)\big)\longrightarrow H^3(\mathfrak{g}).\label{eq2}
\end{gather}
is exact. It is known that $H^3(\mathfrak{g})\cong k$, see \cite[p.~113]{CheE48}. Moreover, by Lemma~\ref{l7},
\begin{gather*}
H^3\big(\mathfrak{g},H^0(p-2,p-2)\big)\cong L(1,1)^{(1)}.
\end{gather*}
Then, the exactness of the sequence~\eqref{eq2} implies that
\begin{gather*}
H^3(\mathfrak{g},L(p-2,p-2))\cong H^3\big(\mathfrak{g},H^0(p-2,p-2)\big)\cong L(1,1)^{(1)},
\end{gather*}
since there is no $G$-homomorphism between the modules $L(1,1)^{(1)}$ and $k$. By~\eqref{is1},
\begin{gather*}H^5(\mathfrak{g},L(p-2,p-2))\cong H^3(\mathfrak{g},L(p-2,p-2))^*\cong L(1,1)^{(1)}.
\end{gather*}

 $(iii)$ In the previous statement, we proved that
\begin{gather*}
H^3(\mathfrak{g},L(p-2,p-2))\cong H^3\big(\mathfrak{g},H^0(p-2,p-2)\big).
\end{gather*}
Then, since the sequence~\eqref{eq12} is exact, the following sequence is exact:
\begin{gather*}
0\longrightarrow H^3(\mathfrak{g})\longrightarrow H^4(\mathfrak{g},L(p-2,p-2))\longrightarrow H^4\big(\mathfrak{g},H^0(p-2,p-2)\big)\longrightarrow H^4(\mathfrak{g}).
\end{gather*}
Since $H^3(\mathfrak{g})\cong k$ and $H^4(\mathfrak{g})=0$, then the sequence
\begin{gather*}
0\longrightarrow k\longrightarrow H^4(\mathfrak{g},L(p-2,p-2))\longrightarrow H^4\big(\mathfrak{g},H^0(p-2,p-2)\big)\longrightarrow 0
\end{gather*}
is exact. By the statement $(iii)$ of Lemma~\ref{l7},
\begin{gather*}
H^4\big(\mathfrak{g},H^0(p-2,p-2)\big)\cong 2L(1,1)^{(1)}\oplus k.
\end{gather*}
There is no $G$-homomorphism between the modules $L(1,1)^{(1)}$ and $k$, so the last exact sequence is split. Then, we get an isomorphism of $G$-modules of the statement $(iii)$.
\end{proof}

Theorem~\ref{th1} follows from Lemmas~\ref{l1}--\ref{l4} and \ref{l8}.

\section{Cohomology with coefficients in Weyl modules}

In this section, we prove Corollaries~\ref{c2} and~\ref{c3}.
Let us start with Corollary~\ref{c2}. We use the following {\it linkage principle} (see \cite[p.~264]{Humph80}):
{\it Let $V$ be indecomposable $G$-module, having $L(\lambda)$ and $L(\mu)$ as composition factors. Then $\lambda$ and $\mu$ are linked.}

Obviously, $H^0(\lambda)$ is peculiar, if it contains a peculiar composition factor.
According to the linkage principle, any composition factor of $H^0(\lambda)$ is linked to $L(\lambda)$. By Theorem~\ref{th1}, there are only six peculiar simple modules. Therefore, $H^0(\lambda)$ is peculiar only in the following cases, which appear in Theorem~\ref{th1}:
\begin{gather*}
\lambda=0,\quad
(p-2)\omega_1+\omega_2,\quad
\omega_1+(p-2)\omega_2,\quad
(p-3)\omega_1,\quad
(p-3)\omega_2,\quad
(p-2)(\omega_1+\omega_2).
\end{gather*}
Let us consider each of these cases separately.

 $(a)$ Obviously, $H^0(0,0)\cong k$. Then, the needed statement follows from the statement $(a)$ of Theorem~\ref{th1}.

Further, we will proceed as in the proof of Lemma~\ref{l8}.

$(b)$ There is the short exact sequence
\begin{gather*}
0\longrightarrow L(p-2,1)\longrightarrow H^0(p-2,1)\longrightarrow L(p-3,0)\longrightarrow 0.
\end{gather*}
Consider the corresponding long cohomological exact sequence
\begin{gather*}
\cdots \longrightarrow H^n(\mathfrak{g},L(p-2,1))\longrightarrow H^n\big(\mathfrak{g},H^0(p-2,1)\big)\longrightarrow
H^n(\mathfrak{g},L(p-3,0))\longrightarrow \cdots.
\end{gather*}
According to Lemmas~\ref{l1}, \ref{l3}, the last long cohomological exact sequence splits into the following exact sequences:
\begin{gather*}
0\longrightarrow H^0\big(\mathfrak{g},H^0(p-2,1)\big)\longrightarrow 0,
\\
0\longrightarrow L(1,0)^{(1)}\longrightarrow H^1\big(\mathfrak{g},H^0(p-2,1)\big)\longrightarrow 0,
\\
0\longrightarrow H^2\big(\mathfrak{g},H^2(p-2,1)\big)\longrightarrow L(1,0)^{(1)}\longrightarrow 0,
\\
0\longrightarrow H^3\big(\mathfrak{g},H^0(p-2,1)\big)\longrightarrow H^3(\mathfrak{g},L(p-3,0))\longrightarrow 2L(1,0)^{(1)}
\\ \hphantom{0}
{}\longrightarrow
H^4\big(\mathfrak{g},H^0(p-2,1)\big)\longrightarrow 0,
\\
0\longrightarrow H^5\big(\mathfrak{g},H^0(p-2,1)\big)\longrightarrow L(1,0)^{(1)}\longrightarrow 0,
\\
0\longrightarrow H^6(\mathfrak{g},H^0(p-2,1))\longrightarrow L(1,0)^{(1)}\longrightarrow H^7(\mathfrak{g},L(p-2,1))
\\ \hphantom{0}
{}\longrightarrow
H^7\big(\mathfrak{g},H^0(p-2,1)\big)\longrightarrow 0,
\\
0\longrightarrow H^8\big(\mathfrak{g},H^0(p-2,1)\big)\longrightarrow 0.
\end{gather*}
The first and last exact sequences yield $H^n\big(\mathfrak{g},H^0(p-2,1)\big)=0$ for $n=0,8$. The second, third, and fifth exact sequences yield isomorphisms
\begin{gather*}
H^1\big(\mathfrak{g},H^0(p-2,1)\big)\cong L(1,0)^{(1)},
\qquad
H^2\big(\mathfrak{g},H^0(p-2,1)\big)\cong L(1,0)^{(1)},
\end{gather*}
and
\begin{gather*}
H^5\big(\mathfrak{g},H^0(p-2,1)\big)\cong L(1,0)^{(1)},
\end{gather*}
respectively.
Consider the fourth exact sequence. Composition factors of $H^3\big(\mathfrak{g},H^0(p-2,1)\big)$ can only be $H^3$ with coefficients in either $L(p-2,1)$ or the socle of $H^0(p-2,1)/L(p-2,1)$.
According to Lemma~\ref{l1}, $H^3(\mathfrak{g},L(p-2,1))=0$. The socle of $H^0(p-2,1)/L(p-2,1)$ is isomorphic to the simple module $L(p-3,0)$. According to Lemma~\ref{l3},
\begin{gather*}
H^3(\mathfrak{g},L(p-3,0))\cong L(1,0)^{(1)}.
\end{gather*}
Therefore,
\begin{gather*}
H^3\big(\mathfrak{g},H^0(p-2,1)\big)\cong L(1,0)^{(1)}.
\end{gather*}
Then, the fourth exact sequence yields an isomorphism
\begin{gather*}
H^4\big(\mathfrak{g},H^0(p-2,1)\big)\cong 2L(1,0)^{(1)}.
\end{gather*}

According to Lemma~\ref{l1}, in the sixth exact sequence, the map
\begin{gather*}
H^7(\mathfrak{g},L(p-2,1))\longrightarrow
H^7\big(\mathfrak{g},H^0(p-2,1)\big)
\end{gather*}
is an epimorphism. Consequently, there are the isomorphisms
\begin{gather*}
H^7(\mathfrak{g},L(p-2,1))\cong
H^7\big(\mathfrak{g},H^0(p-2,1)\big)\cong L(1,0)^{(1)}.
\end{gather*}
Then the sixth exact sequence yields an isomorphism
\begin{gather*}
H^6\big(\mathfrak{g},H^0(p-2,1)\big)\cong L(1,0)^{(1)}.
\end{gather*}

$(c)$ The proof is similar to the previous statement.

$(d)$ Since $H^0(p-3,0)\cong L(p-3,0)$, the statement follows from the statement $(d)$ of Theorem~\ref{th1}.

$(e)$ Since $H^0(0,p-3)\cong L(0,p-3)$, the statement follows from the statement $(e)$ of Theorem~\ref{th1}.

$(f)$ A part of this statement is proved in Lemma~\ref{l7}. We will prove only the rest of the statement. Using the statement $(a)$ of Theorem~\ref{th1} and Lemma~\ref{l8}, and the long cohomological exact sequence \eqref{eq12}, we obtain the following exact sequences:
\begin{gather*}
0\longrightarrow L(1,1)^{(1)}\longrightarrow H^5\big(\mathfrak{g},H^0(p-2,p-2)\big)\longrightarrow k \longrightarrow 0,
\\
0\longrightarrow H^6\big(\mathfrak{g},H^0(p-2,p-2)\big)\longrightarrow 0,
\\
0\longrightarrow k\longrightarrow H^7\big(\mathfrak{g},H^0(p-2,p-2)\big)\longrightarrow 0,
\\
0\longrightarrow H^8\big(\mathfrak{g},H^0(p-2,p-2)\big)\longrightarrow k \longrightarrow 0.
\end{gather*}
These exact sequences yield the isomorphisms
\begin{gather*}
H^5\big(\mathfrak{g},H^0(p-2,p-2)\big)\cong L(1,1)^{(1)}\oplus k,
\\
H^6\big(\mathfrak{g},H^0(p-2,p-2)\big)=0,
\\
H^7\big(\mathfrak{g},H^0(p-2,p-2)\big)\cong k,
\\
H^8\big(\mathfrak{g},H^0(p-2,p-2)\big)\cong k.
\end{gather*}

Using \eqref{is1} and Corollary~\ref{c2} for $H^n(\mathfrak{g},V(\lambda))$, we get Corollary~\ref{c3}.

Corollaries~\ref{c2} and~\ref{c3} show that $H^n\big(\mathfrak{g},H^0(\lambda)\big)\cong H^n(\mathfrak{g},V(\lambda))$, except for the case where \begin{gather*}
 \lambda=(p-2)(\omega_1+\omega_2).
 \end{gather*}

\subsection*{Acknowledgements}

This research is funded by the Science Committee of the Ministry of Education and Science of the Republic of Kazakhstan (grant No~AP08855935).
The author is grateful to the editor and referees whose comments greatly improved the exposition of this paper.

\pdfbookmark[1]{References}{ref}
\LastPageEnding

\end{document}